\documentclass {amsart}

\usepackage{amsthm,amsmath,amssymb,verbatim} 
\usepackage[all]{xy}
\UseComputerModernTips

\theoremstyle{plain} 
	\newtheorem{thm}{Theorem}[section] 
	 
	\newtheorem{lemma}[thm]{Lemma}  
	\newtheorem{cor}[thm]{Corollary} 
 
\theoremstyle{definition}

\newcommand{\aro}{\longrightarrow}

\newcommand{\ep}{\varepsilon}
\newcommand{\id}{\text{id}}
\newcommand{\N}{\mathbb N}
\newcommand{\R}{\mathbb R}

\newcommand{\K}{ \mathcal K}

\newcommand{\f}{ \varphi}

\newcommand{\Lip}{\text{\rm{Lip}}}
\newcommand{\ev}{\text{\rm{ev}}}
\newcommand{\diam}{\text{\rm{diam}}}

\newcommand{\T}{\mathcal T}
\newcommand{\domin}{\preccurlyeq}

\begin{document} 
 
\title	
	{Kernels, distances, and bridges}
 
\author[H. Movahedi-Lankarani]
	{H. Movahedi-Lankarani}
\address{Department of Mathematics\\
	Penn State Altoona\\
         Altoona, PA 16601-3760} 
\email{hxm9@psu.edu}
 
\author [R. Wells]
        {R. Wells}
\address{Department of Mathematics\\
         Penn State University\\
         University Park, PA 16802}

\keywords{almost distance;  bi-Lipschitz embedding; bridge; canonical map; distance;  Lipschitz;  uniform point separation} 
\subjclass{Primary: 54E40, 58C20; Secondary: 54C25, 54F45, 54F50, 58C25, 57R35, 57R40, 26B05}

\date{\today}

\maketitle 

\abstract  
The purpose of this paper is  to study more general real-valued functions of two variables than just metrics on a set $X$.  We  concentrate mainly on the  classes of distances and  almost distances.   We also introduce the notion of a  bridge on the disjoint union of two sets and show that it induces a symmetric distance on the disjoint union.
 \endabstract


\section{Introduction}

This paper is mainly concerned with certain classes of kernels. 
By a  kernel on a set $X$ we mean a  function $\kappa : X \times X \longrightarrow \mathbb R$.    Usually, a subset of the following list of (overlapping)  conditions is imposed on $\kappa$: For all $x, y, z \in X$
\begin{enumerate}
\item[(a)]  $\kappa (x, z) \le \kappa (x, y) + \kappa (y, z)$;\ \ \ (triangle inequality)
\item[(b)]  $0 \le \kappa (x, y)$;
\item[(c)]  $\kappa (x, x) = 0$;
\item[(d)]  $\kappa (x, y) = 0 {\text{ and }} \kappa (y, x) = 0$  imply that $x = y$;
\item[(e)]  $\kappa (x, y) = 0$ implies  $x = y$; 
\item[(f)]  $\kappa (x, y) = \kappa (y, x)$. \ \ \  (symmetry)
\end{enumerate}

\noindent  We note that condition (a) implies that   $0 \le \kappa (x, x)$ for all $x \in X$  so that  $ 0 \le \kappa (x, y) + \kappa (y, x)$, for all $x, y \in X$.   Consequently, conditions (a) and (f) together imply  condition (b).   Also, the two  conditions (a) and (b) imply that for all $x, y, z \in X$ we have 
\[
\bigl\vert \kappa  (x, z) - \kappa  (y, z)\bigr\vert \vee \bigl\vert \kappa  (z, x) - \kappa  (z, y)\bigr\vert \le  \kappa (x, y) \vee  \kappa (y, x).
\] 
\noindent  For a very simple example of a kernel which satisfies conditions (a) -- (d) but not (e), let $\kappa : [0, \infty) \times [0, \infty) \aro \R$ by setting $\kappa (x, y) = (y - x) \vee 0$.   With regards to terminology, a kernel  $\kappa$ is called  a
\begin{enumerate}
 \item[(i)]   {\bf distance} if and only if it satisfies (a) and (b);
\item[(ii)]    weak metric  in \cite{ribeiro} --  see also \cite[page 206]{kelley} -- if and only if it satisfies (a) -- (c);
\item[(iii)]     multimetric   if and only if it satisfies (a) -- (d);
\item[(iv)]   quasi-metric in \cite{quasimet} and  \cite{approach} if and only if   it satisfies (a) -- (e); 
\item[(v)]     pseudometric  if and only if   it satisfies (a),  (b),   (c), and (f);
\item[(vi)]     metric  if and only if   it satisfies (a) -- (f).
\end{enumerate}

\noindent  In \cite{ribeiro}, it is shown that a compact Hausdorff space with a weak metric is metrizable and  the  question is raised as to whether compact Hausdorff hypothesis may be  replaced by normal. This question is answered in the negative in \cite{quasimet}  by constructing a completely normal quasi-metric space which is not metrizable.

\vskip5pt
 
The paper is organized as follows:   Section 2 is mostly about distances and their properties.   An  unexpected, but simple,  result is Theorem \ref{ivt}   which may be viewed as  the inverse function theorem  for bi-Lipschitz embeddings; see also   Lemma \ref{wowlem}.  In Section 3, after generalizing a folklore result  in Lemma \ref{1kapsighat},      almost distances are defined and  some of their basic properties are established.  A symmetric almost distance which is zero on the diagonal is called a  quasimetric in  \cite{juha} and a  b-metric  in \cite{czer}.  Using the Chittenden metrization theorem \cite{chit}, it is shown in \cite{som} that the topology induced by a b-metric is in fact metrizable.  Section 4 is entirely devoted to the notion of a bridge on the disjoint union of two sets, while Section 5 deals with kernels on a probability space satisfying the Hilbert Schmidt condition.  In particular, it is shown that many of the results in \cite[Section 3]{errbilip} for metrics hold true, with very little change,  for distances. The notion of uniform point separation from \cite{bilip} is revisited in Section 6, and we finish the paper by remarking on some related issues in Section 7.


\section{Distances}

   Writing $ \mathcal K (X)$  for the class of all kernels on a set $X$, it is  obvious  that with operations of pointwise addition and multiplication by scalars,   $\mathcal K (X)$ is a linear space over $\R$.  Each kernel  $\kappa \in \mathcal K (X)$ induces a topology on $X$, which we call  the $\kappa$-topology  and denote it by  $\T_{\kappa}$,  as follows:   The topology $\T_{\kappa}$ is the smallest topology in which all functions $\kappa (x, -): X \aro \R$ and $\kappa (-, x): X \aro \R$ are continuous;   for details  see    remark I in Section 7.   Clearly then,  if $\T$  is a topology on $X$ in which every $\kappa (x, -)$ and every $\kappa (-, x)$ is $\T$-continuous, then    $\id: (X, \T) \aro (X, \T_{\kappa})$ is continuous.  Henceforth, we write 
\[
\mathcal K^+ (X) = \{\kappa \in \mathcal K (X) \bigm\vert \kappa (x, y) \ge 0 {\text{ for all }} x, y \in X\}.
\] 

   Given two distances on $X$, their sum, their maximum, or any convex combination of the two is again a  distance.   The following result is folklore; we will  also  prove a more general version  of it for almost distances  in the next section; see Lemma \ref{1kapsighat} and Corollary \ref{genhat}.  For  $\kappa \in \K (X)$, we  write 
\[
X_0 (\kappa) =  \{x \in X \bigm\vert \kappa (x, x) = 0\}.
\]

\begin{lemma}
\label{uniquelargest}
Let $\kappa \in \mathcal K^+ (X)$.  Then there is a unique largest distance $\hat\kappa \le \kappa$.
Moreover,  $\hat\kappa$ is $\T_{\kappa} \times \T_{\kappa}$-continuous on $X_0 (\kappa) \times X_0 (\kappa)$.
\end{lemma}  

\begin{proof}  The construction of $\hat\kappa$ is a classical one:  For $x, y \in X$, we set 
\begin{equation}
\label{kappahat}
\hat\kappa  (x, y) = \inf \, \biggl\{\sum_{i = 0}^{ n - 1} \kappa (z_i, z_{i + 1}) \bigm\vert z_0 = x, \  z_n = y, {\text{ and }}  n \in \N\biggr\},
\end{equation}

 \noindent and it is clear that $\hat\kappa (x, y) \le \kappa (x, y)$.
 
 \vskip5pt  
 
 To see that   $\hat\kappa$ satisfies the triangle inequality, let $x, y, z \in X$ and let $\ep > 0$.  Then there exists a chain $u_0 = x, \cdots, u_n = y$ such that  
 \[
 \sum_{i = 0}^{n - 1} \kappa (u_i, u_{i + 1}) - \hat\kappa (x, y) < \ep / 2.
 \]
 
\noindent Also,  there exists a chain $v_0 = y, \cdots, v_m = z$ such that
 \[
 \sum_{i = 0}^{m - 1} \kappa (v_i, v_{i + 1}) - \hat\kappa (y, z) < \ep / 2.
 \]
 
 \noindent Then $u_0 = x, \cdots, u_n = y = v_0, \cdots, v_m = z$  is a chain from $x$ to $z$ with 
 \[
 \aligned
\hat\kappa (x, z) & \le  \sum_{i = 0}^{n - 1} \kappa (u_i, u_{i + 1})  +  \sum_{i = 0}^{m - 1} \kappa (v_i, v_{i + 1})  \\
& < \hat\kappa (x, y) + \hat\kappa (y, z) + \ep.
\endaligned
\]

\noindent Hence, $\hat\kappa (x, z) \le \hat\kappa (x, y) + \hat\kappa (y, z)$, and $\hat\kappa$ is a distance.

\vskip5pt

We next check that $\hat\kappa \le \kappa$ is the largest such distance.  Let $\lambda \le \kappa$ be a distance.  Then for any $x, y \in X$ and any chain $z_0 = x, \cdots, z_n = y$ we have 
\[
\lambda (x, y) \le \sum_{i = 0}^{n - 1} \lambda (z_i, z_{i + 1} ) \le \sum_{i = 0}^{n - 1} \kappa (z_i, z_{i + 1} ).
\]

\noindent Hence, 
\[
\lambda (x, y) \le  \inf \, \biggl\{\sum_{i = 0}^{ n - 1} \kappa (z_i, z_{i + 1} \bigm\vert z_0 = x, \  z_n = y, {\text{ and }}  n \in \N\biggr\} =\hat\kappa  (x, y).
\]

Finally, we show that $\hat\kappa$ is $\T_{\kappa} \times \T_{\kappa}$-continuous on  $X_0 (\kappa)\times X_0 (\kappa)$.  To this end, let $\ep > 0$ and let $(x, y) \in X_0 (\kappa) \times X_0 (\kappa)$.   Since the function $\kappa (-, x)$  is  $\T_{\kappa}$-continuous, there is a  $\T_{\kappa}$-neighborhood $U$ of $x$ such that if $x^{\prime} \in U \cap X_0 (\kappa)$, then we have $\kappa (x, x^{\prime}) \vee \kappa (x^{\prime}, x) < \ep/2$.  Similarly,  since   the function $\kappa (y, -)$ is $\T_{\kappa}$-continuous, there is a $\T_{\kappa}$-neighborhood $V$ of $y$ such that if $y^{\prime} \in V \cap X_0 (\kappa)$, then $\kappa (y^{\prime}, y) \vee  \kappa (y, y^{\prime}) < \ep/2$.  Hence, if $(x^{\prime}, y^{\prime}) \in \bigl(U \cap X_0 (\kappa)\bigr) \times \bigl(V \cap X_0 (\kappa)\bigr)$, then
\[
\aligned
\hat\kappa (x^{\prime}, y^{\prime}) &\le \hat\kappa (x^{\prime}, x) + \hat\kappa (x, y) + \hat\kappa (y, y^{\prime}) \\
&\le  \kappa (x^{\prime}, x) +  \hat\kappa (x, y) + \kappa (y, y^{\prime})\\
& < \hat\kappa (x , y) + \ep.
\endaligned
\]

\noindent In a similar way, if $(x^{\prime}, y^{\prime}) \in \bigl(U \cap X_0 (\kappa)\bigr) \times \bigl(V \cap X_0 (\kappa)\bigr)$, then   $\hat\kappa (x, y)  < \hat\kappa (x^{\prime} , y^{\prime}) + \ep$.   Consequently,  $(x^{\prime}, y^{\prime}) \in \bigl(U \cap X_0 (\kappa)\bigr) \times \bigl(V \cap X_0 (\kappa)\bigr)$ implies that $\left\vert\hat\kappa (x, y) -  \hat\kappa (x^{\prime} , y^{\prime})\right\vert < \ep$, and $\hat\kappa$ is  $\T_{\kappa} \times \T_{\kappa}$-continuous at $(x, y)$.	 
\end{proof}

The following corollary is  clear.

\begin{cor}
\label{hatmeasur}
The transformation $\xymatrix{ \mathcal K^+ (X)	\ar[r]^{\widehat{}} &\mathcal K^+ (X)}$ is a projection.  Moreover, a  distance  $\kappa$ on a  set  $X$  is $\T_{\kappa} \times \T_{\kappa}$-continuous on $X_0 (\kappa) \times X_0 (\kappa)$.
\end{cor}

Let $\mathcal D \K^+ (X) =  \bigl\{\kappa \in \K^+ (X) \bigm\vert \kappa {\text{ is a distance}}\bigr\}$.    In view of the lemma above, we see that  if $\emptyset \ne S \subset \mathcal D \K^+ (X)$ is bounded above by $\kappa$, then $\hat\kappa = {\text{lub}}\,  S$.   Similarly, if $\emptyset \ne S \subset \mathcal D \K^+ (X)$ is bounded below, then ${\text{ lub}} \, \bigl\{\lambda \in \mathcal D\K^+ (X) \bigm\vert \lambda \le \sigma, \, \sigma \in S\bigr\} = {\text{glb}} \, S$.

\begin{cor}
\label{lattice}
The lattice $\bigl(\mathcal D \K^+ (X), \, \ge\bigr)$ is  a  complete  lattice.
\end{cor}

Now let $\kappa, \f \in \K^+ (X)$.  Then $\kappa\vee \f$ and $\kappa + \f \in \K^+ (X)$ with $(\kappa \vee \f\widehat{)} \ge \hat\kappa \vee \hat\f$ and $(\kappa + \f\widehat{)} \ge \hat\kappa + \hat\f$.  Moreover, if $\f \le \kappa$, then  $\kappa - \f \in \K^+ (X)$ with $(\kappa - \f\widehat{)} \le \hat\kappa - \hat\f$.  Consequently, we have  $\bigl(\kappa - \hat\kappa\widehat{\bigr)} \le \hat\kappa - \hat\kappa = 0$.  This proves the following lemma.

\begin{lemma}
\label{delta-deltahat}
For every  $\kappa \in \mathcal K^+ (X)$  we have  $\bigl(\kappa - \hat\kappa\widehat{\bigr)} = 0$.
\end{lemma}

Let   
\begin{equation}
\label{0ondiagonal}
\aligned
\mathcal K_0 (X) & = \left\{\kappa \in \K (X) \bigm\vert X_0 (\kappa) = X\right\}\\
&= \left\{\kappa \in \mathcal  K (X) \bigm\vert \kappa (x, x) = 0   {\text{ for all }} x \in X\right\}.
\endaligned
\end{equation} 

\noindent  Then  $\mathcal K_0 (X)$ is a linear subspace of $\mathcal K (X)$.  If $\kappa \in \K$, then the kernel $\widetilde\kappa$  defined by  setting $\widetilde\kappa (x, x) = 0$ and $\widetilde\kappa (x, x^{\prime}) = \kappa (x, x^{\prime})$  for  $x \ne x^{\prime}$   is  in $\K_0 (X)$.  However, this is a crude construction and the two topologies $\T_{\kappa}$ and $\T_{\widetilde\kappa}$ are different.  Here is a  perhaps  better way: There is a linear map $Z: \mathcal K (X) \aro \mathcal K_0 (X)$ given by setting  $Z \kappa = \kappa -  \Lambda  \kappa$, where 
\[
\bigl(\Lambda \kappa\bigr) (x, y) = \frac{1}{2} \bigl[\kappa (x, x) + \kappa (y, y)\bigr]. 
\]

\noindent    It is clear that $\bigl(\Lambda \kappa\bigr) (x, x) = \kappa (x, x)$ so that $(Z \kappa) (x, x) = 0$.  Moreover, $\Lambda^2 = \Lambda$ and $Z \Lambda = 0 = \Lambda Z$. It is now easily checked that     $Z^2 = Z$ so that the map $Z$ is a projection.

\vskip5pt

The following corollary  is  immediate as well.

\begin{cor}
\label{pseudo}
If $\kappa \in \mathcal K^+ (X) \cap \mathcal K_0 (X)$  is symmetric, then $\hat\kappa$ is a pseudometric.
\end{cor}

As usual, for $\kappa \in \K (X)$,  we define ${}^t\kappa \in \K (X)$ by setting ${}^t\kappa (x, y) = \kappa (y, x)$.    Obviously, $\kappa \wedge {}^t\kappa \le \kappa \le \kappa \vee {}^t\kappa$.  Moreover,  if  $\kappa$ is a distance, then for any $1 \le p < \infty$, the kernel $(\kappa^p + {}^t \kappa^p)^{1/p}$ is a symmetric distance.  Indeed, if $\kappa$ is a distance, then $\kappa \wedge {}^t\kappa$ is the largest symmetric distance $\le \kappa$, while $\kappa \vee {}^t\kappa$ is the smallest symmetric distance $\ge \kappa$.

\vskip5pt

Now let $\kappa \in \mathcal K (X)$ be a weak metric.      We define a relation $\sim$ on $X$ by setting $x \sim y$ if and only if $\kappa (x, y) = \kappa (y, x) = 0$.  Then $\sim$ is  transitive because $\kappa$ is a distance; that it is symmetric and reflexive is clear.  Moreover, if $x \sim x^{\prime}$ and $y \sim y^{\prime}$, then $\kappa (x^{\prime}, y^{\prime}) = \kappa (x, y)$.  Hence, for two equivalence classes $[x]$ and $[y]$ setting $\bar\kappa \bigl([x], [y]\bigr) = \kappa (x, y)$ is well defined.  We let $\bar X = X/\sim$ and define an epimorphism  $\xymatrix{X \ar@{->>}^{\pi} [r] &\bar X}$ by setting $\pi (x) = [x]$.  Then the kernel $\bar\kappa: \bar X \times \bar X  \aro [0, \infty)$  is a quasi-metric.

\begin{cor}
\label{Xbar}
Let $\kappa \in \mathcal K (X)$  be a weak metric.  Then there exist  a quasi-metric $\bar\kappa$ on   $\bar X$ and an epimorphism $\xymatrix{X \ar@{->>}^{\pi} [r] &\bar X}$ with $\bar\kappa \circ (\pi \times \pi) = \kappa$.
\end{cor}

We finish this section  with the following simple but unexpected (to us) result  which may be viewed as the inverse function theorem for bi-Lipschitz embeddings.  To this end, let $\kappa$ and $\f$ be   distances on  sets  $X$ and $Y$, respectively.  We say that a map $f: X \aro Y$ is  {\bf{Lipschitz}} in  $\kappa$ and $\f$ provided that there is $0 \le  L$ such that 
\[
\f \bigl(f (x), f (x^{\prime})\bigr)  \le L\, \bigl[\kappa (x, x^{\prime}) \vee \kappa (x^{\prime}, x)\bigr],
\]

\noindent for all $x, x^{\prime} \in X$.   In this case, if    $\kappa$ is $\T_{\kappa} \times \T_{\kappa}$-continuous and $\f$ is $\T_{\f} \times \T_{\f}$-continuous, then $f: (X, \T_{\kappa}) \aro (Y, \T_{\f})$ is continuous.   As usual, the Lipschitz constant of $f$ is given by
\[
\Lip (f) = \sup_{x \ne x^{\prime}} \, \frac{\f \bigl(f (x),  f (x^{\prime})\bigr)}{\kappa (x, x^{\prime}) \vee  \kappa (x^{\prime}, x)}.
\]
 
\noindent We say that $f$ is {\bf{bi-Lipschitz}}  in  $\kappa$ and $\f$ provided that $f^{- 1}$ is also Lipschitz  in  $\f$ and $\kappa$.  This is the case precisely when there exist $0 < \ell \le L$ such that
\[
\aligned
 \ell \, \bigl[\kappa (x, x^{\prime}) \vee \kappa (x^{\prime}, x)\bigr] 
 & \le \f \bigl(f (x), f (x^{\prime})\bigr) \\
 & \le L \, \bigl[\kappa (x, x^{\prime}) \vee \kappa (x^{\prime}, x)\bigr]
\endaligned
\]

\noindent  for all $x, x^{\prime} \in X$ We say that $f$ is a  {\bf{formal isometry}}  if $f$ is  bi-Lipschitz in $\kappa$ and $\f$  with $\ell = 1 = L$.  We note  that a formal isometry does not necessarily preserve the distance.  For instance, $\id: (X, \kappa) \aro (X, {}^t\kappa)$ is a formal isometry.  If,  however, we require  isometries  to preserve the distance, then  both $\kappa$ and $\f$ have to be  symmetric.

\vskip5pt

Given two pseudometric spaces $(X, \kappa)$ and $(Y, \sigma)$, we write $\kappa \oplus \sigma$ for any one of the bi-Lipschitz equivalent pseudometrics 
\[
\kappa \vee \sigma \le \left(\kappa^p + \sigma^p\right)^{1/p} \le 2^{1/p}  (\kappa \vee \sigma), \  1 \le p < \infty,
\]

\noindent on $X \times Y$.  

\begin{thm}[Inverse Function Theorem]
\label{ivt}
Let $\kappa$ be a weak metric on a set $X$  and let $(Y, \sigma)$ be a pseudometric space.     Then a map $f: X \aro Y$ is lower Lipschitz if and only if  $\kappa: \left(X \times X, f^{\ast} \sigma  \oplus  f^{\ast} \sigma\right) \aro [0, \infty]$ is Lipschitz.  In this case, we have $\Lip (\kappa) = \Lip \left(f^{- 1}: f (X) \aro X\right)$.  If,  in addition, $f$ is Lipschitz, then $f: (X, \kappa) \aro \bigl(f (X), \sigma\bigr)$  is a bi-Lipschitz equivalence.
\end{thm}  

\begin{proof} For the only if part, assume that there is $L > 0$ such that 
\[
\aligned
\bigl\vert\kappa (x, y) - \kappa (x^{\prime}, y^{\prime})\bigr\vert &\le L \left[f^{\ast}\sigma (x, x^{\prime}) + f^{\ast} \sigma (y, y^{\prime})\right]\\
&= L \left[\sigma \bigl(f (x), f (x^{\prime})\bigr) + \sigma \bigl(f (y), f (y^{\prime})\bigr)\right],
\endaligned
\]

\noindent for every pair of points $(x, y), (x^{\prime}, y^{\prime}) \in X \times X$.  In particular, for $x^{\prime} = y^{\prime} = y$ we have $\kappa (x, y) \le L\, \sigma \bigl(f (x), f (y)\bigr)$.

\vskip5pt

Conversely, let $(x, y), (x^{\prime}, y^{\prime}) \in X \times X$.  By assumption, there exists $0 < \ell$ such that $\ell \, \kappa (x, y) \le  \sigma \bigl(f (x), f (y)\bigr)$.  Since $\kappa$ is a distance,  we have
\[
\bigl\vert\kappa (x, y) - \kappa (x^{\prime}, y^{\prime})\bigr\vert \le \bigl[\kappa (x, x^{\prime}) + \kappa (y^{\prime}, y)\bigr] \vee \bigl[\kappa (x^{\prime}, x) + \kappa (y, y^{\prime})\bigr].
\]

\noindent Hence, 
\[
\ell \, \bigl\vert\kappa (x, y) - \kappa (x^{\prime}, y^{\prime})\bigr\vert \le  \sigma \bigl(f (x), f (x^{\prime})\bigr) + \sigma \bigl(f (y), f (y^{\prime})\bigr)
\]

\noindent because $f^{\ast} \sigma$ is symmetric.   Clearly, $\ell = 1/ L$.
\end{proof}


\section{Almost distances}

  Our goal in this section is to prove a more general version of Lemma \ref{uniquelargest};  we use a variant of a known construction.  To this end, let  $X$ be a set and  let $\kappa, \sigma \in \K (X)$. We say that $\kappa$  is  right dominated   by  $\sigma$, and  write $\kappa  \domin_{_R} \sigma$,  provided that  for all $x, y, z  \in X$,   
\begin{equation}
\label{rtdom}
\kappa (x, z) \le \kappa (x, y) + \sigma (y, z).
\end{equation}

\noindent    In a similar way, we say that $\kappa$  is  left dominated   by  $\sigma$, and write $\kappa \domin_{_L} \sigma$, provided that for all $x, y, z \in X$, 
\begin{equation}
\label{lftdom}
\kappa (x, z) \le \sigma (x, y) + \kappa  (y, z).
\end{equation}

\noindent       It follows from either of (\ref{rtdom}) or (\ref{lftdom}) that,  although $\sigma$ is not assumed to be  a distace,     we have  that $0 \le \sigma (x, x)$ for all $x \in X$ and that   $0 \le \sigma (x, y) + \sigma (y, x)$, for all $x, y \in X$.   It is clear that $\kappa \domin_{_R} \sigma$ if and only if  ${}^t\kappa \domin_{_L} {}^t\sigma$.  More precisely, letting $T: X \times X \aro X \times X$ denote the transposition map $(x, y) \mapsto (y, x)$, we see that  $\kappa \domin_{_R} \sigma$ if and only if  $T \circ \kappa \domin_{_L}  T \circ \sigma$.  Similarly, $\kappa \domin_{_L} \sigma$ if and only if  $\kappa \circ T \domin_{_R} \sigma \circ T$.

\vskip5pt

We set  $X_{\infty} = \coprod_{n \ge 0} X^n$; recall that $X^0 = \{\emptyset\}$.   Given non-negative kernels  $\kappa, \sigma \in \K^+  (X)$, we define a function     $S (\kappa, \sigma) : X \times X_{\infty} \times X \aro [0, \infty)$ by setting  
 \begin{equation}
\label{sum}
S  (\kappa, \sigma) (x, p, y) = \left\{\aligned
&\kappa (x, y), \ {\text{if}} \  p = \emptyset\\
&\kappa (x, x_1) + \sigma  (x_1, y), \ {\text{if}} \ p = (x_1)\\
&\kappa (x, x_1) + \sum_{1 = i}^{ n - 1}  \sigma (x_i, x_{i +1}) + \sigma (x_n, y), \ {\text{if}} \ p = (x_1,  \cdots, x_n), n \ge 2.
\endaligned\right.
\end{equation}

\noindent Clearly, in general  $S (\kappa, \sigma) \ne S  (\sigma, \kappa)$.   Also, for $n \in \N$, we define $\pi : X^n \aro X^{n - 1}$ by setting
\[
\pi (x_1, x_2, \cdots, x_n) = \left\{\aligned
&\emptyset \ {\text{ if }}\  n = 1,\\
&(x_2, x_3, \cdots, x_n) \ {\text{ if }}\  n \ge 2.
\endaligned\right.
\]

\noindent Then (\ref{sum}) may be written as
\[
S (\kappa, \sigma) (x, p, y) = \left\{\aligned
&\kappa (x, y) \ {\text{ if }} \  p = \emptyset,\\
&\kappa (x, x_1) + S  (\sigma, \sigma) (x_1, \pi p, y) \ {\text{ if }} \ p = (x_1, x_2, \cdots, x_n).
\endaligned\right.
\]

\noindent  We set  
\begin{equation}
\label{almost}
 \widehat S ( \kappa, \sigma ) (x, y) = \inf \{S (\kappa, \sigma) (x, p, y)\bigm\vert p \in X_{\infty}\},
\end{equation}

\noindent and we have $\widehat S (\kappa, \sigma)  (x, y) \le \kappa (x, y)$ for all $x, y \in X$.  There are two possible cases.  In the first case $\widehat S (\kappa, \sigma) (x, y) = \kappa (x, y)$ so that we have
\[\widehat S (\kappa, \sigma) (x, z) \le S (\kappa, \sigma) (x, y, z) = \kappa (x, y) + \sigma (y, z) = \widehat S (\kappa, \sigma) (x, y) + \sigma (y, z).
\]

\noindent In the second case, for  any $0 < \ep < \kappa (x, y) - \widehat S (\kappa, \sigma) (x, y)$ there is $\emptyset \ne p (\ep) \in X_{\infty}$ such that  
\begin{equation}
\label{Skapsig}
\widehat S (\kappa, \sigma) (x, y) \le S (\kappa, \sigma)\bigl(x, p (\ep), y\bigr) \le \widehat S (\kappa, \sigma)  (x, y) + \ep.
\end{equation}

\noindent Hence, in this case we have 
\[
\aligned
\widehat S (\kappa, \sigma)  (x, z) & \le S (\kappa, \sigma)\bigl(x, \bigl(p (\ep), y\bigr), z\bigr)\\
&=   S (\kappa, \sigma)\bigl(x, p (\ep), y\bigr) + \sigma (y, z) \le  \widehat S (\kappa, \sigma)  (x, y) + \ep + \sigma (y, z).
\endaligned
\] 

\noindent The following result now holds.

\begin{lemma}
\label{1kapsighat}
We have $\widehat S (\kappa, \sigma)  \domin_{_R} \sigma$.
\end{lemma}

\noindent In the above notation,  Lemma \ref{uniquelargest} may be written as the following corollary.

\begin{cor}
\label{genhat}
 Let $\kappa \in \mathcal K^+ (X)$.  Then there is a unique largest distance $\hat\kappa = \widehat S (\kappa, \kappa)  \le \kappa$.  Moreover, the kernel $\lambda = \kappa \vee {}^t\kappa$ is symmetric and $\hat\lambda =  \widehat S (\lambda, \lambda)$ is the unique largest symmetric distance $\le \lambda$.  Also,  writing $\kappa_0 (x, y) = \kappa (x, y)$ for $x \ne y$ and $\kappa_0 (x, x) = 0$, we see that $(\hat\kappa)_0 = \widehat S (\kappa, \kappa)_0 = \widehat S (\kappa_0, \kappa_0) = \widehat{(\kappa_0)}$.
\end{cor}

Now let $\kappa \domin_{_R} \sigma$,  let $\emptyset \ne p = (x_1, \cdots, x_n) \in X_{\infty}$, and let $x, y \in X$.  Then 
\[
S (\kappa, p, \sigma) \ge \kappa (x, x_n) + \sigma (x_n, y) \ge \kappa (x, y)
\]

\noindent so that  
\begin{equation}
\label{kapsighat}
\widehat S (\kappa, \sigma) (x, y) = \inf \left\{\kappa (x, x^{\prime}) + \sigma (x^{\prime}, y) \bigm\vert x^{\prime} \in X\right\} \ge \kappa (x, y),
\end{equation}

\noindent for all $x, y \in X$.  It  follows from (\ref{almost}) and (\ref{kapsighat}) that if $\kappa \domin_{_R} \sigma$, then $\kappa =  \widehat S (\kappa, \sigma)$.  Also, it is clear that if $\f \domin_{_R} \sigma$ and $\f \le \kappa$, then $\f = \widehat S (\f, \sigma)$.  We collect these in the following lemma.

\begin{lemma}
\label{kaphat}
Let   $\kappa, \sigma \in \K^+ (X)$.   Then there exists a kernel $\widehat S (\kappa, \sigma) \in  \K^+ (X)$ satisfying the following properties:
\begin{enumerate}
\item $\widehat S (\kappa, \sigma) \domin_{_R} \sigma$.
\item  $\widehat S (\kappa, \sigma) = \sup \left\{\f \bigm\vert \f \le \kappa {\text{\rm{ and }}} \f \domin_{_R} \sigma\right\}   \le \kappa$.
\item If $\kappa \domin_{_R} \sigma$, then $\kappa = \widehat S (\kappa, \sigma)$.
\item $\widehat S \bigl(\widehat S (\kappa, \sigma), \sigma\bigr) = \widehat S (\kappa, \sigma) = \widehat S (\kappa, \hat\sigma)$. 
\item $\kappa \domin_{_R} \sigma$ if and only if $\kappa \domin_{_R}  \hat\sigma$.
\end{enumerate}
\end{lemma}

Next, given $q > 0$, we  say that a  kernel $\kappa \in \mathcal K^+ (X)$ is a $q$-{\bf{almost distance}} provided that for all $x, y, z \in X$  we have    
\begin{equation}
\label{almostdist}
 \kappa  (x, z)  \le  q \, \bigl[\kappa (x, y) + \kappa (y, z)\bigr]. 
\end{equation}

\noindent   We say that $\kappa$ is an  almost distance  provided that it is a $q$-almost distance for some $0 < q$.   

\begin{lemma}
\label{bldisaldis}
Let   $\kappa, \sigma \in \K^+ (X)$ satisfy $\ell \,  \sigma \le \kappa \le L \, \sigma$ for some $0 < \ell \le L$ with $\sigma$  a distance.
\begin{enumerate}
\item The kernel $\kappa$ is an $(L/\ell)$-almost distance.
\item  If $L < 1$, then the kernel $\sigma - \kappa$ is a $(1 - \ell)/ (1 - L)$-almost distance.
\end{enumerate}
\end{lemma}

\begin{proof}    Statement (1) is a direct consequence of definitions.    For $(2)$, let $\delta = \sigma - \kappa$.  Then $(1 - L) \,\sigma \le \delta \le (1 - \ell) \, \sigma$, and for all $x, y, z \in X$ we have
\[\aligned
\delta (x, z) & \le (1 - \ell)\, \sigma (x, z) \le (1 - \ell) \bigl[\sigma (x, y) + \sigma (y, z)\bigr]\\
&\le \frac{1 - \ell}{1 - L} \bigl[\delta (x, y) + \delta (y, z)\bigr].
\endaligned
\]
\end{proof}

 Obviously, a distance  is a $q$-almost distance for every $1 \le q$.  Also,    if $\kappa$ is  constant on $X \times X$, then  it satisfies (\ref{almostdist})  for all $q \ge 1/2$.  As in the case of distances,  the sum,  the maximum, or a convex combination of  two   almost distances  is again an almost distance.   Also,  if $\kappa$ is a $q$-almost distance, then for $p \ge 1$ the kernel $\kappa^{1/p}$ is a $q^{1/p}$-almost distance.  Moreover,  given  $\kappa \le \sigma \le \tau  \in \K^+ (X)$, if $\tau - \sigma$ is $q_1$-almost distance and $\sigma - \kappa$ is $q_2$-almost distance, then $\tau - \kappa$ is $[q_1 \vee q_2]$-almost distance.
The following lemma is elementary.

\begin{lemma}
\label{aldistdom}
 Let   $\kappa \in \K^+ (X)$.
 \begin{enumerate}
 \item If $\kappa \domin_{_R}  q \kappa$ and $\kappa \domin_{_L}  q  \kappa$ for some $q > 0$, then $\kappa$ is a $(q + 1)/2$-almost distance.
 \item If $\kappa$ is a $q$-almost distance for some $q > 1/2$, then $\kappa \domin_{_R}  (2q - 1) \kappa$ or $\kappa \domin_{_L}  (2q - 1)  \kappa$.
  \item  The kernel $\kappa$ is a distance if and only if $\kappa \domin_{_R} \kappa$ or $\kappa \domin_{_L} \kappa$.
 \item    Let $\kappa$ be a $q$-almost distance for some $q \ge 1/2$.  Then for each triple of points $x, y, z \in X$ we have   
 \begin{equation}
\label{aldist1}
 \bigl\vert\kappa (x, z) - \kappa (y, z)\bigr\vert \le 2 q \,[\kappa (x, y) \vee \kappa (y, x)] 
\end{equation} 

\noindent or  
\begin{equation}
\label{aldist2}
 \bigl\vert\kappa (x, z) - \kappa (x, y)\bigr\vert \le 2 q \,[\kappa (z, y) \vee \kappa (y, z)].
\end{equation}
 \end{enumerate}
\end{lemma}

A word of caution is in order here.  It is possible that for a given triple of points $x, y, z \in X$ only one (but not both) of (\ref{aldist1}) or (\ref{aldist2}) holds, even if $\kappa$ is symmetric.  For instance, let $X = \{x, y, z\}$ with $\kappa (x, z) = \kappa (z, x) = 2$, $\kappa (y, z) = \kappa (z, y) = 1$, and $\kappa (x, y) = \kappa (y,x) = 0$.  Then $\kappa$ is a $2$-almost distance which satisfies (\ref{aldist2}) but not (\ref{aldist1}).  Furthermore,  a kernel $\kappa \in \K^+ (X)$ satisfying both (\ref{aldist1}) and  (\ref{aldist2}) for all $x, y, z \in X$ is not necessarily an almost distance.   This motivates the following definition.  We say that  $\kappa \in \K^+ (X)$ is a {\bf uniform  $q$-almost distance}  for some $q \ge 1/2$ provided that $\kappa$ is a $q$-almost distance and {\it{both}} (\ref{aldist1}) and (\ref{aldist2}) hold for  {\it{every}}  $x, y, z \in X$.  We note that for a symmetric uniform $q$-almost distance, the two inequalities (\ref{aldist1}) and (\ref{aldist2}) are equivalent.  Moreover, if a kernel $\kappa \in \K^+ (X)$ is symmetric and satisfies both (\ref{aldist1}) and  (\ref{aldist2}) for all $x, y, z \in X$, then it is a $(q +1)/2$-almost distance, and hence, a uniform almost distance.


\section{Bridges}  

   If there exist metrics $d_X$ and $d_Y$ on two sets $X$ and $Y$, respectively, which agree on $X \cap Y$, then  we may define a metric $d$ on $X \cup Y$ by setting $d\big\vert_X = d_X$, $d\bigl\vert_Y = d_Y$, and  
\begin{equation}
\label{union}
d (x, y) = \inf \left\{ d_X (x, z) + d_Y (z, y) \bigm\vert z \in X \cap Y\right\}.
\end{equation}

\noindent In the  case that $X \cap Y = \emptyset$, (\ref{union}) becomes
\[
d (x, y) = \inf \left\{ d_X (x, z) + d_Y (z, y) \bigm\vert z \in \emptyset\right\} = +\,  \infty.\]

\noindent In this case, of course, $d$ is not a metric.   If $X \cap Y = \emptyset$  and $\diam (X, d_X) \vee \diam (Y, d_Y) < \infty$, then $d (x, y) = \diam (X, d_X) \vee \diam (Y, d_Y)$ works. But, what about the general case?

\vskip5pt

As pointed out in \cite{flood}, this is the reason that the category of metric spaces and contractions does not have coproducts.  One way to overcome this shortcoming is to enlarge the category to include the  extended metric\footnote{The metric takes values in $[0, \infty) \cup \{\infty\}$.}  spaces and set $d (x, y) = +  \infty$ for $x \in X$ and $y \in Y$.  Another approach  is to attach a  {\bf norm function}  to each metric space:  Following  \cite{flood} and \cite{pestov}, a  norm function  on a metric space $(X, d_X)$ is a function $f : X \aro \R$ satisfying
\begin{equation}
\label{norm}
\bigl\vert f (x_1) - f (x_2)\bigr\vert  \le d _X (x_1, x_2) \le f (x_1) + f (x_2),\end{equation}

\noindent for all $x_1, x_2 \in X$.  It follows immediately that  such   $f \ge 0$ and  that, for  $x \in X$,
\[
f (x) = \sup \,  \bigl\{d_X (x, z) - f (z) \bigm\vert z \in X\bigr\}.
\]
  For instance, for each $x \in X$, the function $d_X (x, -)$ is a norm function on $(X, d_X)$. Choosing a norm function $f$ on $(X, d_X)$ is similar to attaching a point $\ast$ to $X$  and setting $d (x, \ast) = f (x)$.   Strictly speaking, one first identifies $X$ with its isometric image inside the injective  hull of $X$, \cite{isbell}  and \cite{lang}, and then attaches the point $\ast = f$ to it.  In this case, the triple $(X, d_X, f)$ is called a  normed metric space in \cite{flood} and a  normed set  in \cite{pestov}.  

\vskip5pt

Now let $(X, d_X, f)$ and $(Y, d_Y, g)$ be normed sets.  Then the triple  $(X\amalg Y, d, f\amalg g)$, where  $d\big\vert_X = d_X$, $d\big\vert_Y = d_Y$,  $d (x, y) = f (x) + g (y)$, and $(f \amalg g) (x) + (f \amalg g) (y) = f(x) + g (y)$  is also a normed set.  In a similar way, the triple $(X \times Y, d = d_X + d_Y, f \otimes g)$, where $(f \otimes g) (x, y) = f (x) + g (y)$, is again  a normed set.  Indeed, it is shown in \cite{flood} that the category of normed sets and contractions is small complete and small cocomplete.  More generally, it is shown in \cite{pestov} that this category   is both complete and  cocomplete.

\vskip5pt

What about other norm functions on $X \times Y$ or $X \amalg Y$?  For instance, if $\omega$ is norm function on $(X \times Y, d_X + d_Y)$, it must satisfy 
\begin{equation}
\label{phionprod}
\bigl\vert\omega (x, y) - \omega (x^{\prime}, y^{\prime})\bigr\vert \le d_X (x, x^{\prime}) + d_Y (y, y^{\prime}) \le \omega (x, y) + \omega (x^{\prime}, y^{\prime}).
\end{equation}

\noindent The difficulty, in general, is with the right hand side of (\ref{phionprod}). The  construction in the above paragraph works because $\omega (x, y) = f(x) + g (y)$ which is a very special case.

\vskip5pt

Here is a more general construction.  Let $X$ and $Y$ be sets.   A {\bf bridge} on the disjoint union $X \amalg Y$ is a function $F: X \times Y \aro \R$ satisfying
\begin{equation}
\label{bridge}
F (x_1, y_1) - F (x_2, y_2)  \le F (x_1, y_2) + F (x_2, y_1)
\end{equation}

\noindent for all $x_1, x_2 \in X$ and all $y_1, y_2 \in Y$.  It is clear that a bridge $F \ge 0$.

\vskip5pt

A bridge $F$ on $X \amalg Y$ induces two non-negative symmetric kernels $\kappa_{_X}$ and $\kappa_{_Y}$ on $X$ and $Y$, respectively, by setting
\[
\kappa_{_X} (x_1, x_2) = \inf \left\{F(x_1, y) + F (x_2, y) \bigm\vert y \in Y\right\}
\]
and
\[
\kappa_{_Y} (y_1, y_2) = \inf \left\{F(x, y_1) + F (x, y_2) \bigm\vert x \in X\right\}.
\]
Then, by Lemma \ref{uniquelargest}, there exist unique  symmetric distances $\hat\kappa_{_X} \le \kappa_{_X}$ and  $\hat\kappa_{_Y} \le \kappa_{_Y}$, respectively.  The following lemma is clear.{\footnote{For a very simple example, let $c > 0$. Then the constant function $F (x, y) = c$  is a bridge on $X \amalg Y$ inducing the symmetric distance $\kappa (x_1, x_2) = 2 c = \kappa (y_1, y_2)$ and $\kappa (x, y) = c = \kappa (y, x)$.}

\begin{lemma}
\label{BridgeDist}
Let $X$ and $Y$ be sets and let $F: X \times Y \aro \R$ be a bridge on the disjoint union   $X \amalg Y$.  Then
\[
\kappa  = \hat\kappa_{_X} \amalg F \amalg {}^t F \amalg \hat\kappa_{_Y}
\]
is a symmetric distance on   $X \amalg Y$.
\end{lemma}

 Given a bridge $F$ on $X \amalg Y$,   for each $x \in X$ the function $F (x, -) : Y \aro \R$ satisfies  
\[
F (x, y_1) - F (x, y_2) \le \kappa_{_Y} (y_1, y_2) \le  F (x, y_1) + F (x, y_2),
\]

\noindent a norm function on the set $(Y, \kappa_{_Y})$.
 Similarly, for each $y \in Y$, the function  $F (-, y) : X \aro \R$ satisfies 
\[
F (x_1, y) - F (x_2, y) \le  \kappa_{_X} (x_1, x_2) \le F (x_1, y) + F (x_2, y),
\]
 
\noindent a norm function on the set $(X, \kappa_{_X})$.

\vskip5pt

As mentioned above, if $(X, d_X, f)$ and $(Y, d_Y, g)$   are  normed sets, then the function $F (x, y) = f (x) + g (y)$ is a bridge on $X \amalg Y$.  We will refer to such a bridge as   Flood-Pestov bridge.  

\vskip5pt

Now let $X$ and $Y$ be sets and let $F : X \times Y \aro \R$ be a bridge on $X \amalg Y$.  Let $\left(x_{\circ}, y_{\circ}\right) \in X \times Y$ be fixed.  Then for $(x, y) \in X \times Y$ we have
\[
F (x, y) \le  F (x_{\circ}, y_{\circ})  +  F (x, y_{\circ}) + F (x_{\circ}, y).
\]

\noindent Setting $f (x) = F (x_{\circ}, y_{\circ})  +  F (x, y_{\circ})$ and $g (y) = F (x_{\circ}, y)$, it follows from simple calculation that the function $G (x, y) = f (x) + g (y)$ if a Flood-Pestov bridge on $X \amalg Y$ with $F (x, y) \le G (x, y)$ for all $(x, y) \in X \times Y$.

\begin{lemma}
\label{F-P}
Let $X$ and $Y$ be sets and let $F: X \times Y \aro \R$ be a bridge on the disjoint union   $X \amalg Y$.  Then there exists a Flood-Pestov bridge  $G$  on   $X \amalg Y$ with $F (x, y) \le G (x, y)$ for all $(x, y) \in X \times Y$.
\end{lemma}


\section{Kernels on a probability space}

In this section we extend some some of the results for metrics in \cite{errbilip} to the case of distances.

\vskip5pt

Let  $(X, \mu)$  be a probability  space  and let $\kappa \in \K (X)$.   Then  for a Lebesgue measurable subset $A \subset \R$ and $x \in X$, the sets $\kappa (x, -)^{- 1} (A)$ and $\kappa (-, x)^{- 1} (A)$ are $\mu$-measurable.  Consequently, $\kappa: (X \times X, \mu \times \mu) \aro  (\R, \,  {\text{Borel}} \, \sigma{\text{-algebra}})$ implying that the topology  $\T_{\kappa}$ is contained in the $\sigma$-algebra of $\mu$-measurable subsets of $X$ so that $\T_{\kappa}$-Borel subsets of $X$ are $\mu$-measurable.  We say that $\kappa$ is $\mu$-{\bf{regular}} provided that every non-empty $\T_{\kappa}$-open set has positive $\mu$-measure.

\vskip5pt

We write $\K (\mu)  =  \K (X, \mu)$  for the linear subspace of $\K (X)$  consisting of all kernels on $X$ which are $\mu \times \mu$-measurable  and  satisfy    the Hilbert-Schmidt condition
\[
\int\int \kappa^2 (x, y)  d \mu (x)  d \mu (y) < \infty.
\]
	
\noindent   Then for $\mu$-almost every $x \in X$ the functions $\kappa (x, -)$ and $\kappa (-, x)$ are in $L^2 (\mu)$.   It is clear that there is a surjection $\xymatrix{\mathcal K (\mu) \ar@{->>}[r] & L^2 (\mu \times \mu)}$.  Moreover, the triple $\bigl(\K (\mu), +, \ast\bigr)$ is a ring where  the multiplication\footnote{That the multiplication $\ast$ is associative  follows from the theorem of   Fubini.}  $\ast: \K (\mu) \times \K (\mu) \aro  \K (\mu)$ is given by 
\begin{equation}
\label{star}
\aligned
(\kappa \ast \f) (x, y)  &= \int \kappa (x, z)\, \f (z, y) \, d \mu (z) \\
& = \langle\kappa (x, -), \f (-, y)\rangle,
\endaligned
\end{equation} 

\noindent  the inner product in $L^2 (\mu)$. Clearly,  if $\kappa$  is symmetric,   so is $\kappa \ast \kappa$.

\vskip5pt

The ring $\bigl(\K (\mu), +, \ast\bigr)$ acts on $L^2 (\mu)$ by  ``convolution'' as follows:  For $\kappa \in \K (\mu)$ and $f \in L^2 (\mu)$ we have 
\begin{equation}
\label{action}
(\kappa \centerdot f) (x) = \int \kappa (x, z) f (z) \, d \mu (z)  
=  \langle\kappa (x, -),  f \rangle.
\end{equation}

\noindent If $(X, \mu)$ is a measurable group  \cite[Section 59]{halmos}  and $f \in L^2 (\mu)$, we let $\kappa (x, y) = f (x y^{- 1})$.  Then for  $g \in L^2 (\mu)$, the function $\kappa \centerdot g$ is just the usual convolution of $f$ and $g$.  In particular, if $X$ is  abelian and locally compact with $\mu$ the Haar measure,  then 
\[(\kappa \centerdot g)(x) = \int f (x - y)\,  g (y) \, d \mu (y).\]

Let $\K_0 (\mu) = \K (\mu)  \cap \K_0 (X)$  and $\K^+ (\mu) = \K (\mu)  \cap \K^+ (X)$.  We  recall that each $\kappa \in \K_0 (\mu)$ induces a canonical map $\iota_{\kappa}: X \aro L^2 (\mu)$ by setting $\iota_{\kappa} (x) = \kappa (x, -)$.   For a  weak metric $\kappa \in \K_0 (\mu)$ we have
\[
\bigl\vert\Vert\iota_{\kappa}\Vert^2_2 - \kappa (x, y) \, \Vert\iota_{\kappa}\Vert_1 \, \bigr\vert \le (\kappa \ast \kappa) (x, y) \le \Vert\iota_{\kappa}\Vert^2_2  + \kappa (x, y) \, \Vert\iota_{\kappa}\Vert_1
\]

\noindent so that $(\kappa \ast \kappa) (x, x) = \Vert\iota_{\kappa} (x)\Vert^2_2 \ge 0$.  Moreover, if  $\kappa$ is a metric, then $\iota_{\kappa}$ is injective.

\vskip5pt

The canoincal map $\iota_{\kappa}$ induces a pseudometric  $\rho_{\kappa}$ on $X$ by setting   
\begin{equation}
\label{iota^ast}
\rho_{\kappa} (x, y) =  \iota_{\kappa}^{\ast}\bigl(\Vert\cdot - \cdot\Vert_2\bigr) (x, y) = \Vert\iota_{\kappa} (x) - \iota_{\kappa} (y)\Vert_2.
\end{equation}

\noindent  Clearly,  if $\kappa \in \mathcal K^+ (\mu)$  is a distance, then $\rho_{\kappa} (x, y) \le \kappa (x, y) \wedge \kappa (y, x)$ for all $x, y \in X$.  Moreover, if $\kappa$  is a pseudometric, then for every $x \in X$ both functions $\rho_{\kappa} (x, -): (X, \kappa) \aro [0, \infty)$ and $\rho_{\kappa} (-, x): (X, \kappa) \aro [0, \infty)$ are Lipschitz    with Lipschitz constant $ \le 1$.  Hence, in this case, all $\rho_{\kappa}$-Borel sets are  $\mu$-measurable.  In general, however, the difficulty with $\rho_{\kappa}$ is that we do not know whether $\rho_{\kappa}$-Borel sets are  $\mu$-measurable or not.  Interestingly,  for  any symmetric  kernel $\kappa \in \mathcal K^+ (\mu)$ we have
\[\aligned
Z (\kappa \ast \kappa) (x, y) &= (\kappa \ast \kappa) (x, y) - \Lambda (\kappa \ast \kappa) (x, y)\\
 &= - \frac{1}{2} \left\Vert\iota_{\kappa} (x) - \iota_{\kappa} (y)\right\Vert^2_2  =  - \frac{1}{2} \,\rho^2_{\kappa} (x, y),
\endaligned
\]

\noindent see Section 2.  The  following is an immediate corollary of Theorem \ref{ivt}.

\begin{cor}
\label{wowlem}
Let $\kappa \in \mathcal K^+ (\mu)\,  \cap \, \mathcal K_0 (\mu)$  be a distance and let $\sigma \in \K^+ (\mu)$. Then the function $\kappa: \left(X \times X, \rho_{\sigma} \oplus \rho_{\sigma}\right) \aro [0, \infty)$ is Lipschitz if and only if the  map $\iota_{\sigma} : (X, \kappa) \aro L^2 (\mu)$ is lower Lipschitz.  In this case,  we have  $\Lip (\kappa) = \Lip \bigl(\iota^{- 1}_{\sigma}: \iota_{\sigma} (X)  \aro X\bigr)$.
\end{cor}

   Now, when the action of the ring  $\bigl(\K (\mu), +, \ast\bigr)$  on $L^2 (\mu)$  is restricted to distances, more can be said.   Specifically, most of the results in \cite[Section 3]{errbilip} for metrics hold, almost verbatim,  for distances.  In detail, a distance  $\kappa$ induces a linear transformation $T_{\kappa}: L^2 (\mu) \aro L^2 (\mu)$ by setting 
$
\left(T_{\kappa} f \right) (x) =  (\kappa \centerdot f) (x) =  \langle\iota_{\kappa} (x), f\rangle,
$
\noindent for $ f \in L^2 (\mu)$.  Clearly $T_{\kappa}$ is a Hilbert-Schmidt operator and hence compact and self-adjoint.   Moreover, we have the commutative diagram
 
\begin{equation}
\label{k*k}
\xymatrix{ & L^2 (\mu) \ar[d]^{T_{\kappa}} \\
X \ar[ru]^{\iota_{\kappa}} \ar[r]_{\iota_{\kappa\ast\kappa}} & L^2 (\mu).}
\end{equation}

\noindent We next   write $\Lip_{\kappa} (X)$ for the class of all real-valued Lipschitz functions on $(X, \kappa)$; it is a normed linear space with the norm given by $\Vert f\Vert_{\Lip_{\kappa}} = \sup_{x \in X} \bigl\vert f (x)\bigr\vert \vee  \Lip_{\kappa} (f)$.   The distance  $\kappa$ induces a linear transformation $J_{\kappa}: L^2 (\mu) \aro \Lip_{\kappa} (X)$ by setting\footnote{The distinction between $T_{\kappa}$ and $J_{\kappa}$ is in the codomain.}
\[
\left(J_{\kappa} f \right) (x) = \int \kappa (x, z) f (z) d \mu (z)
\]

\noindent for $ f \in L^2 (\mu)$.  That $J_{\kappa} f \in \Lip_{\kappa} (X)$ follows from 
\[
\bigl\Vert J_{\kappa} f\bigr\Vert_{\Lip_{\kappa}}  \le \bigl[ 1 \vee \sup_{(x, y) \in X \times X} \kappa (x, y) \bigr] \bigl\Vert f\bigr\Vert_2
\]

\noindent so that the operator norm $\Vert J_{\kappa}\Vert \le   1 \vee \sup_{(x, y) \in X \times X} \kappa (x, y)   <  \infty$.    If $\kappa \in \mathcal K_0 (\mu)$ is a bounded   distance, the linear map $J_{\kappa}$ is  bounded, and thus, Lipschitz.  We have the commutative diagram 
\begin{equation}
\label{incl.compact}
\xymatrix{
 & \Lip_{\kappa} (X) \ar@{^{(}->}[d] \\
L^2 (\mu) \ar[ru]^{J_{\kappa}} \ar[r]^{T_{\kappa}} & L^2 (\mu).
}
\end{equation}

\noindent It is  well known that the evaluation map $\ev: X \aro \Lip^{\ast}_{\kappa} (X)$ is a bi-Lipschitz (in fact, isometric when $\kappa$ is symmetric; see \cite{pestov}) embedding.    Letting $\eta: L^2 (\mu) \aro L^2 (X)^{\ast}$ denote the isomorphism $\eta (f) = \langle f, - \rangle$, we have 
\begin{equation}
\label{ev&eta}
\bigl(\eta \circ \iota_{\kappa}\bigr) (x) (f)   = \bigl( J^{\ast}_{\kappa} \circ \ev\bigr) (x) (f),
\end{equation}

\noindent where  $J^{\ast}_{\kappa}: \Lip^{\ast}_{\kappa} (X) \aro L^2 (\mu)^{\ast}$ is the adjoint of $J_{\kappa}$.  Since (\ref{ev&eta}) holds for every $x \in X$ and every $f \in L^2 (\mu)$, we have the commutative diagram 
\begin{equation}
\label{CanEmbEv}
\xymatrix{
X \ar[r]^{\iota_{\kappa}} \ar@{^{(}->}[d]_{\ev} & L^2 (\mu) \ar[d]^{\eta}\\
 \Lip^{\ast}_{\kappa} (X) \ar[r]^{J_{\kappa}^{\ast} }& L^2 (\mu)^{\ast}.
}
\end{equation}




\section{More on uniform point separation}

Now let $(X, d, \mu)$ be a metric-measure space.  That is, $(X, d)$ is a  metric space and $\mu$ is a   Borel regular measure on  $X$ which is non-trivial on nonempty open sets.  We assume, throughout,  that  $0 < \mu (X) < \infty$.   Let $\ep > 0$. For $x, y \in X$ we set 
\begin{equation}
\label{Exy}
E (x, y, \ep; d) = \left\{z \in X \bigm\vert \bigl\vert d (x, z) - d (y, z)\bigr\vert \ge \ep \, d (x, y)\right\}.
\end{equation}

\noindent As in \cite{bilip}  (see also \cite{errbilip}), we say that a   regular Borel measure $\mu$ on $X$   separates points uniformly  with respect to the metric $d$ provided that there exist $\ep > 0$ and $c > 0$ such that $\mu \bigl(E (x, y, \ep; d)\bigr) \ge c$ for all $x, y \in X$.  It is shown in  \cite{bilip}  that for $(X, d, \mu)$ compact, the uniform separation  condition is equivalent to requiring the canonical map $\iota_d: X \aro L^2 (\mu)$   to be bi-Lipschitz.   We recall  from Section 5 that the canonical map $\iota_d$ induces a pseudometric  $\rho_{d}$ on $X$ by setting  $\rho_{d} (x, y) = \Vert\iota_{d} (x) - \iota_{d} (y)\Vert_2$.  The following result is an  immediate corollary of  \cite[Theorem 4.1]{bilip} and  Corollary  \ref{wowlem} above.

\begin{thm}  
\label{bL-sep} 
Let $(X, d, \mu)$ be a compact metric measure space.      The following statements are equivalent:
\begin{enumerate}
\item  The measure $\mu$ separates points uniformly with respect to $d$.
\item  The   canonical  map $\iota_d: (X, d) \aro L^2 (\mu)$ is bi-Lipschitz.
\item   The map $d: \bigl(X \times X, \rho_d \oplus \rho_d\bigr) \aro [0, \infty)$ is Lipschitz.  
\end{enumerate} 
\end{thm}

Now let $(X, d, \mu)$ be a metric-measure space and assume that $\mu$ separates points uniformly with respect to $d$.  Let $\tau$ be a metric on $X$, bi-Lipschitz equivalent to $d$:  There exist $0 < \ell \le L$ such that $\ell \, d \le \tau \le L\,  d$.  Let $\delta = L\, d - \tau \le (L - \ell) \, d$ and  assume that $\delta$   is a uniform $q$-almost distance for some $q > 0$; see (\ref{almostdist}).   For  $x\ne y \in X$ and $z \in E (x, y, \ep; d)$ we compute
\[
\aligned
\ep \le \frac{\bigl\vert d (x, z) - d (y, z)\bigr\vert}{d (x, y)}  & \le \frac{\bigl\vert \tau (x, z) - \tau (y, z)\bigr\vert}{
L\, d (x, y)} + \frac{\bigl\vert \delta (x, z) - \delta (y, z)\bigr\vert}{L\, d (x, y)} \\
& \le \frac{\bigl\vert \tau (x, z) - \tau (y, z)\bigr\vert}{\rho (x, y)} + q \, (1 - \frac{\ell}{L}).
\endaligned
\]

\noindent If $\ell /L > 1 - \ep/ q$, then $\ep  - q \, (1 - \ell/ L) > 0$, so that 
\[
E(x, y, \ep; d) \subset E \left(x, y,  \left[\ep -  q \, \bigl(1 - \frac{\ell}{L}\bigr)\right]; \tau\right)
\]

\noindent and hence, $\mu\left(E \bigl(x, y,  \ep - q\,(1 - \ell /L); \tau\bigr)\right) \ge c$.   This  proves  the following lemma.

\begin{lemma}
\label{drhosepts} 
Let $(X, d, \mu)$ be a metric-measure space and assume that $\mu$ separates points uniformly with respect to $d$.  Let $\tau$ be a metric on $X$ satisfying  $\ell \, d \le \tau \le L\,  d$ for some $0 < \ell \le L$.  If the kernel $\delta = L\, d - \tau$ is a uniform almost distance   and if $\ell /L$ is sufficiently close to $1$,  then the measure $\mu$ separates points uniformly with respect to $\tau$.
\end{lemma}

Given two metrics $d$ and $\rho$ on $X$, we say that $\rho$ is a  dilation\footnote{It is also called  a homothety  in the literature.} of $d$ provided that $\rho = r \, d$ for some  $r > 0$.  The following corollary is immediate.  

\begin{cor}
\label{dilate} 
Let $(X, d, \mu)$ be a metric-measure space and let  $\rho$ be a dilation of $d$.   Then the measure $\mu$ separates points uniformly with respect to $d$  if and only if it separates points uniformly with respect to $\rho$.
\end{cor}

 As in \cite{errbilip}, we let  $\mathcal G (d)$ denote the class of all metrics on $X$ which are bi-Lipschitz equivalent to $d$.  We define a pseudometric $u$ on $\mathcal G  (d)$ as follows\footnote{This pseudometric $u$  on $\mathcal G (d)$ is different from the pseudometric $W_d$ given in \cite{errbilip}.}:  Given $\kappa$ and $\sigma$ in $\mathcal G  (d)$, we let 
\[
L (\sigma, \kappa) = \inf \{L\bigm\vert \kappa \le L\,  \sigma\} = \Lip \bigl(\id : (X, \sigma) \aro (X, \kappa)\bigr)
\]

\noindent and  $\ell (\sigma, \kappa) = \sup \{\ell \bigm\vert \ell \,\sigma \le \kappa\}$ so that  $\ell (\sigma, \kappa) \, \sigma \le \kappa \le L (\sigma, \kappa)\, \sigma$,   and we  set
\[
u (\sigma, \kappa) = \log \frac{L (\sigma, \kappa)}{\ell (\sigma, \kappa)}.
\]

\noindent Clearly  $u (\sigma, \kappa) \ge 0$ with $u (\sigma, \kappa) = 0$ if and only if $\sigma$ is a dilation of $\kappa$.  That $u (\kappa, \sigma) = u (\sigma, \kappa)$  is also clear.  The triangle inequality $u (\sigma, \kappa) \le u (\sigma, \f) + u (\f, \kappa)$ follows from
\[
\ell (\sigma, \f) \, \ell (\f,  \kappa) \le \ell (\sigma, \kappa) \le \frac{\sigma}{\kappa} = \frac{\sigma}{\f} \cdot \frac{\f}{\kappa} \le L (\sigma, \kappa) \le L (\sigma, \f) \, L (\f, \kappa).
\]

\noindent  By replacing $\kappa$ with the dilation $\kappa^{\prime} = \kappa / L (\sigma, \kappa)$,  we see that the bi-Lipschitz relation $\ell (\sigma, \kappa) \, \sigma \le \kappa \le L (\sigma, \kappa) \, \sigma$ is equivalent to  $\exp [- u (\sigma, \kappa^{\prime})] \, \sigma \le \kappa^{\prime} \le \sigma$.     The following is a  corollary of Lemma \ref{bldisaldis} and Lemma \ref{drhosepts}. 

\begin{cor}
\label{drho2} 
Let $(X, d, \mu)$ be a metric-measure space and assume that $\mu$ separates points uniformly with respect to $d$.  Let $\rho \in \mathcal G (d)$ satisfy  $L (\rho, d)  < 1$.    If  $u (d, \rho)$ is sufficiently small,  then the measure $\mu$ separates points uniformly with respect to $\rho$.
\end{cor}

  The pseudometric $\rho_d$ induced on $X$ by the canonical map    $\iota_d: X \aro L^2 (\mu)$  satisfies $\ell (d, \rho_d) \, d \le \rho_d \le d$ so that      $\iota_d$ is bi-Lipschitz if and only if $\ell (d, \rho_d) > 0$. 
 Now, again as in \cite{errbilip},  let  $\mathcal E (d)$ denote the (possibly empty) subset of  $\mathcal G (d)$ consisting of the metrics whose corresponding canonical maps      $X \aro L^2 (\mu)$   are bi-Lipschitz.   The following result is a  consequence of Corollary \ref{drho2}.
 
 \begin{cor}
\label{sigmaphi} 
Let $(X, d, \mu)$ be a compact metric-measure space and assume that $\sigma \in \mathcal E (d)$.  Let $\f \in \mathcal G (d)$ satisfy $L (\f, \sigma)  < 1$.    If  $u (\sigma, \f)$ is sufficiently small,  then $\f \in \mathcal E (d)$.
\end{cor}


\section{Remarks}

\noindent  {\bf{I. }} Given  a kernel  $\kappa \in \mathcal K (X)$,  a basis for  the  smallest topology in which all $\kappa (x, -)$, $x \in X$, are continuous is the family $\left\{U^R [\ep, x] (y) \bigm\vert \ep > 0, x \in X, y \in X\right\}$, where
 \[
 \aligned
 U^R [\ep, x] (y) &= \kappa (x, -)^{- 1} \bigl(\kappa (x, y) - \ep, \kappa (x, y) + \ep\bigr)\\
 &=  \left\{z \in X \bigm\vert \vert\kappa (x, z) - \kappa (x, y) \vert < \ep\right\}.
 \endaligned
  \]

\noindent Similarly, a basis for the smallest topology in which all $\kappa (-, x)$, $x \in X$, are continuous is the family $\left\{U^L [\ep, x] (y) \bigm\vert \ep > 0, x \in X, y \in X\right\}$, where
 \[
 \aligned
 U^L [\ep, x] (y) &= \kappa (-, x)^{- 1} \bigl(\kappa (y, x) - \ep, \kappa (y, x) + \ep\bigr)\\
 &=  \left\{z \in X \bigm\vert \vert\kappa (z, x) - \kappa (y, x) \vert < \ep\right\}.
 \endaligned
  \]

\noindent Then the family $\left\{U^R [\ep, x] (y),  U^L [\ep, x] (y) \bigm\vert  \ep > 0, x \in X, y \in X\right\}$ is a sub-basis for the $\kappa$-topology $\T_{\kappa}$.  

\vskip5pt

\noindent {\bf II. }  
 We recall the following discussion from \cite{bilip}.   Let $(X, \kappa, \mu)$ be a compact metric-measure space. Without loss of generality, we     assume further  that $\mu (X) = 1$.   Then the canonical map $\iota_{\kappa}: X \aro L^2 (\mu)$ lifts to $\varTheta_{\kappa}: X \aro \Lip_{\kappa} (X)$ by setting $\varTheta_{\kappa} (x) = \kappa (x, -)$.   This lift $\varTheta_{\kappa}$ is totally discontinuous with $\varTheta_{\kappa} (X)$ metrically discrete: For $x, y \in X$, we have
$
\left\Vert\varTheta_{\kappa} (x) - \varTheta_{\kappa} (y)\right\Vert_{\Lip_{\kappa}} =  2  \vee \kappa \, (x, y).
$    Hence, if $(X, \kappa)$ is a compact  uncountable metric space,  then the set $\varTheta_{\kappa} (X)$ is closed, uncountable and discrete in the norm topology of $\Lip_{\kappa} (X)$.  In particular, $\Lip_{\kappa} (X)$ is not separable.

  \vskip5pt

It follows from simple calculation that $\Vert\iota_{\kappa\ast\kappa} (x) - \iota_{\kappa\ast\kappa} (y)\Vert_2 \le \diam (X) \, \kappa (x, y)$  for $x, y \in X$.    Moreover, the lift  $\varTheta_{\kappa\ast\kappa}: X \aro \Lip_{\kappa} (X)$ of $\iota_{\kappa\ast\kappa}$ is Lipschitz. Then  diagram (\ref{k*k}) may be amended to give the commutative diagram 
\begin{equation}
\label{liftcan}
\xymatrix{
& & L^2 (\mu) \ar[dd]^{J_{\kappa}} \ar[dl]^{T_{\kappa}} \\
X \ar@/^1.1pc/[urr]^{\iota_{\kappa}} \ar[r]^{\iota_{\kappa\ast\kappa}} \ar@/_1.0pc/[drr]^{\varTheta_{\kappa\ast\kappa}} & L^2 (\mu) \\
&& \Lip_{\kappa} (X). \ar@{_{(}->}[ul]
}
\end{equation}

 The following lemma is from \cite{errbilip}.

\begin{lemma} \cite[Lemma 6.3]{errbilip}
\label{kstark}
Let $(X, \kappa, \mu)$ be a  metric-measure space.   Then the canonical map $\iota_{\kappa\ast\kappa}: X \aro L^2 (\mu)$    is injective.  Furthermore, if $\diam (X) < \infty$, then $\iota_{\kappa\ast\kappa}$ is Lipschitz.
\end{lemma}

\vskip5pt  

\noindent {\bf III. }  Let $(X,  d, \mu)$ be a compact metric-measure space.  In  diagram (\ref{k*k}),  since $d$ is a metric, the map $\iota_{d \ast  d}$ is  one-one and Lipschitz   by Lemma \ref{kstark}.   Moreover,  since  $\iota_{d}$ is Lipschitz and $T_{d}$ is compact,  the next result follows from  \cite[Lemma 2.4 and Theorem 2.5]{bilip}; it is a restatement of \cite[Theorem 6.1]{errbilip}.

\begin{thm}
\label{d*d} 
Let  $(X, d, \mu)$ be a compact metric-measure space.  If the canonical map $\iota_{d \ast d}: X \aro L^2 (\mu)$ is lower Lipschitz, then   there exists a bi-Lipschitz embedding  $\xymatrix{X\,  \ar@{^{(}->}[r] &\R^N}$  for some $N \in \N$.
\end{thm}

   Letting  $\kappa = d$ and $\sigma = d \ast d$ in  Theorem  \ref{ivt},  we see that the canonical map $\iota_{d\ast d}: (X, d) \aro L^2 (\mu)$ is lower Lipschitz if and only if  the map $d: \bigl(X \times X, \rho_{d\ast d} \oplus \rho_{d\ast d}\bigr) \aro [0, \infty)$ is Lipschitz.  This gives the following corollary of Theorem \ref{d*d}.

\begin{cor}
\label{wowcor}
Let $(X, d, \mu)$ be a compact metric-measure space.  If the map $d: \bigl(X \times X, \rho_{d\ast d} \oplus \rho_{d\ast d}\bigr) \aro [0, \infty)$ is Lipschitz, then  there exists a bi-Lipschitz embedding  of $(X, d)$ into some $\R^N$.
\end{cor}

\vskip5pt

\noindent  {\bf{IV. }} Let $(X, d, \mu)$ be a metric-measure space with $0 < \mu (X) < \infty$ and let $\f : X \aro [0, \infty)$ by setting 
\[
\f (x) = \frac{1}{\mu (X)} \int_X d (x, z) \, d \mu (z).
\]

\noindent Then $\bigl\vert \f (x) - \f (y)\bigr\vert \le d (x, y) \le \f (x) + \f (y)$ so that $\f$ is a  norm function on $(X, d)$; see Section 4.   Moreover, its average value $\bar\f \le 2 \inf  \{\f (x) \bigm\vert x \in X\}$. 

\vskip5pt

If $(X, d)$  is the interval $[0, 1]$ with the Euclidean metric and $\mu$  the Lebesgue measure, then   $\f (x) = x^2 - x + 1/ 2$.  However, there are spaces for which the function $\f$ is constant.  For instance, if $(X, d)$  is the unit circle with arc length, then $\f \equiv \pi /2$.   If,  instead, the unit circle  is equipped with  the Euclidean metric inherited from $\R^2$, then $\f \equiv 4/\pi$.    

\vskip5pt

Now suppose that there is a surjective  isometry $h : (X, d) \aro (X, d)$ that satisfies $\mu (E) = \mu \bigl(h^{-1} (E)\bigr)$ for every $\mu$-measurable subset $E \subset X$. In this case, we say that $h$  preserves the measure $\mu$.  Given $x_1, x_2 \in X$ with $h (x_1) = x_2$, we have 
\[
\aligned
\f (x_1) &= \frac{1}{\mu (X)} \int_X d (x_1, z) \, d \mu (z) = \frac{1}{\mu (X)} \int_X d \bigl(h^{-1} (x_2),  h^{-1} (w)\bigr) \, d \mu \bigl(h^{-1} (w)\bigr) \\
&= \frac{1}{\mu (X)} \int_X d (x_2, w) \, d \mu (w) = \f (x_2).
\endaligned
\]

\noindent  Let $\mathcal J (X, d, \mu)$ denote the group of surjective  isometries of $(X, d)$ which preserve the measure $\mu$.

\begin{lemma}
If   the   action of the group $\mathcal J (X, d, \mu)$  on $X$ is  transitive, then the function $\f \equiv  {\text{constant}}$.
\end{lemma}



\begin{thebibliography}{99}



 
 \bibitem{quasimet} M. Balanzat, {\em On the metrization of quasi-metric spaces},  Gaz. Mat. Lisboa {\bf{12}}, no. 50  (1951),  91--94.
 

\bibitem{chit} E. W. Chittenden, {\em  On the equivalence of \'ecart and voisinage},  Trans. Amer. Math. Soc. {\bf 18} (1917), 161--166.

\bibitem{czer}  S. Czerwik, {\em Nonlinear set-valued contraction mappings in b-metric spaces},  Atti Semin. Mat. Fis. Univ. Modena {\bf 46} (1998), 263--276.

\bibitem{flood} J. Flood, {\em Free topological vector spaces}, Dissertationes Math. (Rozprawy Mat.) {\bf 221} (1984), 95 pp.

  
 \bibitem{halmos} P. Halmos, {\em Measure Theory}, van Nostrand, 1950.

\bibitem{juha} J. Heinonen,  {\em  Nonsmooth calculus},  Bull. Amer. Math. Soc.   {\bf 44} (2007),  163--232.

\bibitem{isbell}   J. R. Isbell, {\em Six theorems about injective metric spaces}, Comment. Math. Helv. {\bf 39} (1964), 65--76.
 

  \bibitem{kelley} J. Kelley, {\em General Topology}, van Nostrand, 1955.  
  
  

\bibitem{lang} U. Lang, {\em Injective hulls of certain discrete metric spaces and groups}, J. Topol.  Anal. {\bf 5} (2013), 297--331.

\bibitem{approach}  R. Lowen, {\em Approach Spaces:  The Missing Link in the Topology--Uniformity--Metric Triad},  Oxford Univ. Press, 1997.
  
  
 
  
 
 \bibitem{bilip}  H. Movahedi-Lankarani and R. Wells, {\em On bi-Lipschitz embeddings},  Portugalia. Math. (n. s.) {\bf{62}}  (2005),  247--268.

\bibitem{errbilip} --------------, {\em Bi-Lipschitz embeddings revisited}, arXiv:2501.07648v1 [math.MG] 13 Jan 2025.
 
 \bibitem{pestov}  V. Pestov, {\em Douady's conjecture on Banach analytic spaces},  C. R. Acad. Sci. Paris S\'er I Math.  {\bf{319}}  (1994),  1043--1048.
 
 \bibitem{ribeiro} H. Ribeiro, {\em Sur les espace \`a m\'etrique faible},  Portugalia. Math.  {\bf{4}}  (1943),  21--40.

\bibitem{som} S. Som, {\em Metrizability of b-metric spaces  and $\theta$-metric spaces via Chittenden's metrizatin theorem}, arXiv: 1909.07693v1, 17 September 2019.
 
 \end{thebibliography}
\end{document}